\documentclass{aptpub}
\usepackage{tikz}
\usepackage{pgflibraryarrows}
\authornames{J. Mairesse and P. Moyal} 
\shorttitle{Stochastic matching model}
\usepackage{graphicx}
\usepackage[T1]{fontenc}
\usepackage[latin1]{inputenc}
\usepackage{enumerate}
\usepackage{latexsym,amssymb,amsxtra}
\usepackage{amsmath}
\usepackage{stmaryrd}
\usepackage{pstricks}
\usepackage{pst-node}
\usepackage{multido}
\usepackage{tikz} 
\usepackage{pgf} 
{\bf}{\it}

\def\/{\, | \,}

\def\N{{\mathbb N}}
\def\Z{{\mathbb Z}}

\newcommand{\cM}{{\mathcal M}^+}

\def\v{{\--}}
\def\pv{{\not\!\!\--}}

\newcommand{\R}{{\mathbb R}}

\newcommand{\UnN}{{\mathcal V}}

\def\esp#1{{\mathbf E}\left[#1\right]}

\def\Z{{\mathbb Z}}
\def\I{{\mathbb F}}
\def\I{{\mathbb I}}

\newcommand\tend{{\underset{n \rightarrow \infty}{\longrightarrow}}}

\newcommand\Halmos{\rule[0.4pt]{4pt}{4pt}}

\newcommand\maI{{\mathcal I}}

\newcommand\maE{{\mathcal E}}

\newcommand\maV{{\mathcal V}}
\newcommand\maU{{\mathcal U}}

\def\eref#1{(\ref{#1})}

\newsavebox{\fmbox}
\newenvironment{fmpage}[1]
     {\begin{lrbox}{\fmbox}\begin{minipage}{#1}}
     {\end{minipage}\end{lrbox}\fbox{\usebox{\fmbox}}}

{\end{list}}%

\newenvironment{itemi}
{%
  \begin{list}{$\bullet$}%
  {\noindent%
    \usecounter{enumi}%
    \setlength{\topsep}{2pt}%
    \setlength{\partopsep}{0pt}%
				\setlength{\itemsep}{2pt}%
    \setlength{\parsep}{0pt}%
    \setlength{\leftmargin}{2.5em}%
    \setlength{\labelwidth}{1.5em}%
    \setlength{\labelsep}{0.5em}%
    \setlength{\listparindent}{0pt}%
    \setlength{\itemindent}{0pt}%
  }%
}%
{\end{list}}%
\begin{document}

\title{Stability of the stochastic matching model}

\authorone[CNRS, UPMC]{Jean Mairesse} 
\addressone{CNRS, LIP6 UMR 7606, Sorbonne Universit\'es, UPMC Univ Paris 06, Paris, France}
\authortwo[UTC and Northwestern University]{Pascal Moyal} 
\addresstwo{LMAC, Universit\'e de Technologie de Compi\`egne, Compi\`egne,  
France and\\  
IEMS Department,  
Mc Cormick School of Engineering, Northwestern University, Evanston, USA}

\begin{abstract}
We introduce and study a new model that we call the {\em matching
  model}. 
Items arrive one by one in a buffer and
depart from it as soon as possible but by pairs. The items of a departing pair are said to be
{\em matched}. There is a finite set of classes $\maV$ for the
items, and the allowed matchings depend on the classes, according to a
{\em matching graph} on $\maV$.  
Upon arrival, an item may find several possible matches in the
 buffer. This indeterminacy is resolved by a {\em matching policy}. 
When the sequence of classes of the arriving items is i.i.d., the sequence of
buffer-contents is a Markov chain, whose stability is investigated. In
particular, we prove that the model may be stable if and only if the
matching graph is non-bipartite.
\end{abstract}


\section{{\bf Introduction}}

A {\em matching model}, as described in the abstract, is formally
specified by a triple $(G,\Phi,\mu)$ formed by: 
\begin{itemi}
 \item a {\em matching graph} $G=(\maV,\maE)$, that is, an undirected
   graph whose vertices $\maV$ are the classes of items and whose
edges $\maE$ are the allowed matchings between classes; 
\item a {\em matching policy} $\Phi$ which defines the new
  buffer-content given the pair formed by the old buffer-content and the
  arriving item; 
\item a probability $\mu$ on $\maV$, the common law of the
  i.i.d. classes of the arriving items. 
\end{itemi}

The sequence of buffer-contents forms a Markov chain. The stability
problem consists in determining the conditions on $(G,\Phi,\mu)$ for
the Markov chain to be positive recurrent.

As such, despite being simple and natural, the matching model seems to be 
original. It has a queueing model flavor, with the crucial
specificity that items play the roles of both customers and
servers. 
In spirit, it is related to the general models of ``constrained queueing networks''~\cite{TaEp}, ``input-queued
cross-bar switches''~\cite{MMAW}, or ``call centers with skills-based
routing''~\cite[Section 5]{GKM}.

The present model can be seen as a particular case in discrete time, of the {\em matching queues} introduced 
in \cite{GurWa}, where items may be matched by groups of more than two, and where a control is performed to minimize the holding cost, 
allowing to keep 'matchable' jobs in line, in order to wait for a more profitable match in the future. However, our approach is widely different in that we consider a fixed matching policy, which prohibits the type of control studied in \cite{GurWa}.

The closest connection with existing models in the literature has to be made with the 
recent ``bipartite matching model'' (BM). This connection plays a central role in several proofs. The BM has been introduced in~\cite{CKW09}, see also \cite{AdWe,ABMG}. In this context, items arrive by pairs in a
buffer and depart from it, as soon as possible, also by pairs. There is a
finite number of classes partitioned into ``customer'' classes and
``server'' classes. Each pair, arriving or departing, is formed by
exactly one customer and one server. For departing pairs, an
additional requirement is that the customer and the server should be
matched, with the allowed matchings depending on the classes only.  The
sequence of classes of arriving items is i.i.d. and, in each arriving pair, the
customer is independent of the server. In \cite{BGMa12}, the same
model is studied without the restriction that the arriving customer
and server should be independent. For convenience, let us denote this
last model by EBM (extended BM). 

Clearly, the (E)BM model and the matching model are close. In fact, 
the matching model may be viewed as a particular
case of the EBM model. Indeed, consider a 
matching model with graph $(\maV, \maE)$ and sequence of arriving items
$(v_n)_n$. Let $\widetilde{\maV}$ be a disjoint copy of
$\maV$. Define a bipartite matching model with customer classes
$\maV$, server classes $\widetilde{\maV}$, possible matches
$\{(u,\tilde{v}) \mid (u,v)\in \maE\}$, and arriving sequence
$(v_n,\widetilde{v}_n)_n$. If the matching policies are the same,
then, at any time, 
the buffer-content of the bipartite matching model is
$(U,\widetilde{U})$ if the buffer-content of the original matching
model is $U$. In this bipartite matching model, there is a perfect
correlation between the arriving customer and server, so this is indeed
an EBM model and not a BM model. 

Due to the above connection, we can transfer 
several results proved for the EBM in \cite{BGMa12} to the matching
model. But, on the other hand, we are able to get more precise results in the present context.

\medskip

{\bf Content.} 
Isolating the matching model as an interesting
object of study is the first contribution of the present paper. 
The
second contribution is to show that the matching model may be stable
if and only if the matching graph is non-bipartite (Theorem
\ref{th-main-alea}-\eref{eq-non-bip}). 

In a nutshell, the situation
is as follows. A connected graph $G$ is either bipartite or
not. In the first case, we may construct a stable {\em bipartite}
matching model on $G$ (see \cite{BGMa12}) but not a stable matching model. In
the second case, we may construct a stable matching model on $G$ (and the
bipartite matching model is not even defined). 
Additional results are provided for matching
models on a non-bipartite matching graph: $(i)$ the model is always stable under the natural conditions for
the ``match the 
longest'' policy (Theorem \ref{th-main-alea}-\eref{eq-non-bip-ml});
$(ii)$ this is not true for all matching policies (Proposition \ref{pr-stabAB}). 
This last result is reminiscent of queueing systems which do not achieve their full capacity region, see for instance the model 
with re-entrant lines in \cite{Kum93}. 
The result $(i)$ on the optimality of  ``match the 
longest'', has connections with the result in Tassiulas $\&$
Ephremides \cite{TaEp} stating that in their ``constrained queueing network'', the ``max-weight'' policy  has a maximal stability region. 
Also, the proofs have the same flavor, as they both use a quadratic Lyapunov function.

\medskip

{\bf Convention.}
By default, a {\em graph} is finite simple and undirected, that is, of the form $G=(\maV,\maE)$, 
with $0<\# \maV <\infty$ and 
$\maE\subset (\maV\times \maV)\setminus
\{(v,v), v\in \maV\}$, with $(u,v)\in \maE$ $\implies$ $(v,u)\in \maE$.  Write $u \v v$ for $(u,v) \in \maE$
and $u \pv v$ for $(u,v) \not\in \maE$. 
For $U \subset \maV$, define 
\[
U^c = \maV \setminus U, \qquad \maE(U) = \{ v \in \maV \mid \exists u \in U, \ u
\-- v\}\:.
\]
For $u\in V$, write $\maE(u)=\maE(\{u\})$. For $U\subset \maV$, the {\em
  subgraph induced by} $U$ is the graph $(U,\maE\cap (U\times U))$. 
%
An {\em independent set} of a graph $G$ is a non-empty subset $\maI \subset \maV$ 
which does not include any pair of neighbors, {\em i.e.} : $\bigl(\forall i\neq j \in \maI, \ i \pv j\bigr)$. 
Let $\I$ be the set of independent sets of $G$. 

Given a finite set $S$, denote by $\cM(S)$ the set of probability
measures $\mu$ on $S$ such that for all $i$ in $S$, $\mu(i)>0$. 

Denote by $\N$ the set of non-negative integers. 
Let $A^*$ be the set of finite words over the alphabet $A$. For any
word $w \in A^*$ and
any  letter $a\in A$, let $|w|_a$ be the number
of occurrences of $a$ in $w$. Let $|w|=\sum_{a\in A} |w|_a$ be the
{\em length} of $w$. Let $[w]:=(|w|_a)_{a\in A}\in \N^A$ be the {\em commutative
image} of $w$. 


\section{{\bf The matching model}}
\label{sec:intro}

%
%
The {\em matching model} associated with a
  graph $G$, called the {\em matching graph}, is defined as follows. 
Start with an empty ``buffer'' and, for any $n$ in $\N$, draw an
element $v_n$ of $\maV$ and apply the following rule:  $(i)$ if there is no
element $j$ of $\UnN$  in the buffer such that $v_n \--j$, then add $v_n$ to the buffer; 
$(ii)$ otherwise, do not add $v_n$ and
remove from the buffer an element $j$ such that $v_n\-- j$ (we say that $v_n$ and $j$ are {\em matched} 
together). If several elements $j$ of the buffer are such that $v_n\-- j$, the one to be
removed depends on a {\em matching policy} to be specified. 

\medskip

The sequence $(v_n)_{n\in \N}$ is assumed to be independent and
identically distributed (i.i.d.). 
Throughout the paper, we denote by $\mu$ the common law over $\UnN$ of the
elements $v_n,\,n\in\N$. We always assume that $\mu \in \cM(\maV)$.

\medskip

The stability problem of the matching model can be described in the following
rough terms: what are the conditions on $G$, the matching policy, and
the distribution $\mu$ such that the system is stable, in the sense that the buffer reaches an
equilibrium behavior?

\begin{example}\rm \label{ex-theexample}
Consider the matching graph $G=(\maV,\maE)$ with $\maV=\{1,2,3,4\}$ and 
$\maE =\{(1,2),(2,3),(2,4),(3,4)\}$, see Figure \ref{fi-example}. 
\begin{figure}[h!]
\begin{center}
\begin{tikzpicture}
\draw[-] (2,3) -- (2,2);
\draw[-] (2,2) -- (1,1);
\draw[-] (2,2) -- (3,1);
\draw[-] (1,1) -- (3,1);
\fill (2,3) circle (2pt) node[right] {\small{1}} ;
\fill (2,2) circle (2pt) node[right] {\small{2}} ;
\fill (1,1) circle (2pt) node[below] {\small{3}} ;
\fill (3,1) circle (2pt) node[below] {\small{4}} ;
\end{tikzpicture}
\vspace*{-0.3cm}
\caption[smallcaption]{The matching graph of Example \ref{ex-theexample}.}
\label{fi-example}
\end{center}
\end{figure}
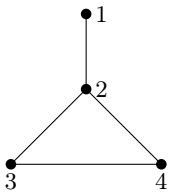

Consider the sequence $(v_n)_n
= 3,1,4,2,4,4,1,2,3,2,1,4,\dots$. Denote by $U_n$ the ordered sequence
of elements in the buffer after
the arrivals of $v_0,\dots, v_{n-1}$. We have
\[
U_1 = (3), \ U_2= (3,1), \ U_3 = (1), \ U_4 = \emptyset,
\ U_5 = (4\}, \ U_6 = (4,4), \ U_7 = (4,4,1), 
\]
and $U_8$ depends on the matching policy: indeed $v_7=2$ can be
matched either with $v_4=4$, $v_5=4$, or with $v_6=1$. A convenient way of
visualizing the dynamic is given in Figure \ref{fi-evolution} (assuming that $v_7$
is matched with $v_6$),

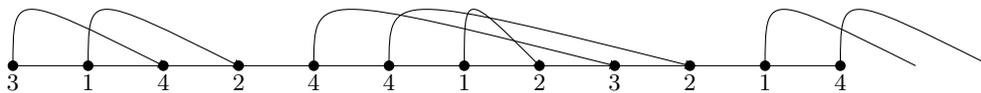
\begin{figure}[htb]
\begin{center}
\begin{tikzpicture}
\draw[-] (0,2) -- (11,2);
\fill (0,2) circle (2pt) node[below] {\small{3}};
\draw[->, thin] (0,2) .. controls +(up:1cm)  .. (2,2);
\fill (1,2) circle (2pt) node[below] {\small{1}};
\draw[->, thin] (1,2) .. controls +(up:1cm)  .. (3,2);
\fill (2,2) circle (2pt) node[below] {\small{4}};
\fill (3,2) circle (2pt) node[below] {\small{2}};
\fill (4,2) circle (2pt) node[below] {\small{4}};
\draw[->, thin] (4,2) .. controls +(up:1cm)  .. (8,2);
\fill (5,2) circle (2pt) node[below] {\small{4}};
\draw[->, thin] (5,2) .. controls +(up:1cm)  .. (9,2);
\fill (6,2) circle (2pt) node[below] {\small{1}};
\draw[->, thin] (6,2) .. controls +(up:1cm)  .. (7,2);
\fill (7,2) circle (2pt) node[below] {\small{2}};
\fill (8,2) circle (2pt) node[below] {\small{3}};
\fill (9,2) circle (2pt) node[below] {\small{2}};
\fill (10,2) circle (2pt) node[below] {\small{1}};
\draw[-, thin] (10,2) .. controls +(up:1cm)  .. (12,2);
\fill (11,2) circle (2pt) node[below] {\small{4}};
\draw[-, thin] (11,2) .. controls +(up:1cm)  .. (13,2);
\end{tikzpicture}
\caption[smallcaption]{The matching model in action, on the matching graph of Figure \ref{fi-example}.} 
\label{fi-evolution}
\end{center}
\end{figure}
\end{example}


\section{{\bf Structural properties of the matching graph}}
\label{sec:structure}

The conditions \textsc{Ncond}, defined below, will turn out to be necessary for the stability of the matching
model (Proposition \ref{thm:mainmono} hereafter). 
This justifies a thorough study of these conditions, which is the purpose of this section. 

\medskip

Let $G=(\maV,\maE)$ and let $\mu \in \cM(\maV)$. Define the following conditions on $\mu$~:


 \noindent
 \begin{center}
 \begin{fmpage}{8.5cm}
$ \qquad \textsc{Ncond}(G): \qquad \forall \maI \in \I, \ \ \  \mu(\maI) \ < \
 \mu\left(\maE(\maI)\right). $
 \end{fmpage}
 \end{center}
\medskip

\noindent We first observe the following, 

\begin{lemma}
\label{lemma:equivNCOND}
For any connected graph $G$ and $\mu \in \cM(\maV)$, {\sc Ncond}$(G)$ is equivalent to 
\begin{equation}
\label{eq:equivNCOND}
\forall U \subset \maV, \ U \ne \emptyset, \ U \ne \maV, \qquad 
\mu(U) \ < \ \mu\left(\maE(U)\right).
\end{equation}
\end{lemma}

\medskip

\proof 
Fix $\mu \in \cM(\maV)$. It is clear that (\ref{eq:equivNCOND}) entails {\sc Ncond}$(G)$, let us
focus on the converse. 
Consider $U \subset \maV, \ U \ne \emptyset, \ U \ne \maV$, such that $U$
is not an independent set. Notice that this implies in particular that $|U|>1$. 

(i) Assume first that the subgraph induced by $U$ is connected. Then,  
$\forall u\in U, \exists v\in U, \ u \-- v$. 
This implies that $U\subset \maE(U)$. Also, since $G$ is connected and 
$U\neq \maV$, we have that $U\varsubsetneq \maE(U)$. Therefore, $\mu(U) < \mu(\maE(U))$. 

(ii) Assume now that the subgraph induced by $U$ has several connected
components. Let $U = U_1\,\cup\, U_2$, where $U_1$ is the union of the connected components of cardinality 1, 
and $U_2$ is the union of the other connected components. The set $U_2$ is non-empty, 
otherwise $U$ would be an independent set. Moreover, $\forall u\in U_2, \exists v\in U_2, \ u \v v$, 
hence exactly as in case (i), we have that
\begin{equation}
\label{eq:connexcomp1}
U_2 \varsubsetneq \maE\left(U_2\right). 
\end{equation}
Now, if $U_1$ is empty, we have $U_2=U$ and (\ref{eq:connexcomp1}) allows to conclude. If not, 
$U_1$ is an independent set and from {\sc Ncond}$(G)$, we get 
$\mu(U_1) < \mu\left(\maE\left(U_1\right)\right)$. 
Also, (\ref{eq:connexcomp1}) entails that $\bigl( \maE(U_1)\cup U_2\bigr) \subset \maE(U)$, and since   
by definition, $\maE(U_1)\cap U_2 =\emptyset$, we obtain that
\[
\mu(U) = \mu(U_1) + \mu(U_2) < \mu(\maE(U_1)) + \mu(U_2) =
\mu(\maE(U_1)\cup U_2)  \leq \mu(\maE(U)) \:,
\]
which concludes the proof. \Halmos
\endproof

With some abuse, let us denote by \textsc{Ncond}$(G)$, the subset of probability measures   
$\mu \in \cM(\maV)$ satisfying the condition \textsc{Ncond}$(G)$.

\begin{figure}[h!]
\vspace{-0.4cm}
\begin{center}
\begin{tikzpicture}
\fill (0,0) circle (1.5pt) node[below left]{0} ;
\draw[->] (0,0) -- (0,3.5);
\fill (0,3.5) node[above]{\scriptsize{$\mu(2)$}} ;
\fill (0,3) circle (1.5pt) node[left]{\scriptsize{$1$}};
\draw[->] (0,0) -- (3.5,0);
\fill (3.5,0) node[right]{\scriptsize{$\mu(1)$}} ;
\fill (3,0) circle (1.5pt) node[below]{\scriptsize{$1$}};
\draw[-] (3,0) -- (0,3);
\draw[-] (0,0) -- (1.5,1.5);
\draw[thick,dashed] (0,0) parabola[bend at end] (1.5,1.5);
\filldraw [fill=gray!60,draw=black]
(0,0) -- (1.5,1.5) -- (0,1.5) -- cycle;
\fill (0,1.5) circle (1.5pt) node[left]{${1\over 2}$};
\end{tikzpicture}
\vspace{-0.4cm}
\caption[smallcaption]{In gray, the projection of the region $\textsc{Ncond}(G)\cap
   \{\mu(3)=\mu(4)\}$.}\label{fi-ncond}
\end{center}
\end{figure}
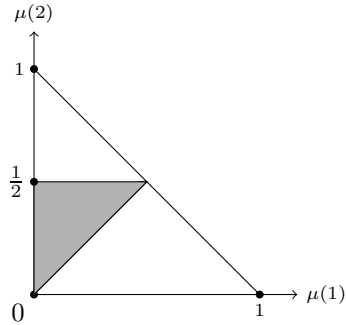
\begin{example}\rm
For the matching graph of Figure \ref{fi-example}, the set of independent sets is
$\I = \bigl\{ \{1\}, \{2\},\{3\},\{4\}, \{1,3\}, \{1,4\}\bigr\}$.  
Therefore, as the total mass of $\mu$ is 1, we have
\begin{equation}\label{eq-ncond}
 \textsc{Ncond} = \bigl\{  \mu(1) < \mu(2) < 1/2, \ \ \  
 \mu(1) + \mu(3) < 1/2, \ \ \ 
 \mu(1) + \mu(4) < 1/2 \bigr\} \:. 
\end{equation}

\noindent Making the simplifying assumption $\mu(3)=\mu(4)$, we get 
$$\textsc{Ncond} \cap \{\mu(3)=\mu(4)\} = \{ \mu(1) < \mu(2) < 1/2
\},$$ 
see Figure \ref{fi-ncond}. 
\end{example}

\paragraph{{\bf Specific conditions for bipartite graphs.}}

Assume that $G=(\maV,\maE)$ is bipartite and let $\maV=\maV_1\cup \maV_2$ be a bi-partition of the vertices in two 
independent sets.  
The following condition $\text{\textsc{Ncond}}_{1/2}(G)$ on the probability measures $\mu \in \cM(\maV)$, 
will be a useful tool in {several} proofs~: 

 \medskip

 \noindent\begin{center}\begin{fmpage}{12cm}
 $\text{\textsc{Ncond}}_{1/2}(G):  \ \  \mu(\maV_1)=\mu(\maV_2)=1/2, \quad \forall \maI \in \I\setminus\{\maV_1,\maV_2\}, \quad \mu(\maI) <
 \mu\left(\maE(\maI)\right).$
 \end{fmpage}\end{center}
\medskip
\medskip
\noindent The set of measures $\text{\textsc{Ncond}}_{1/2}(G)$ is defined likewise {\sc Ncond}$(G)$.

\paragraph{{\bf Bipartite double cover}}

Given a graph $G=(\maE,\maV)$, its {\em bipartite double cover} (see {\em e.g.} \cite{BHM80}) is 
the bipartite graph $2\circ G=(2\circ \maV,2\circ \maE)$ defined by
\begin{equation}\label{eq-double}
2\circ \maV = \maV \ \cup \ \bigl\{\tilde{u} \mid u\in \maV \bigr\}, \qquad 2\circ \maE = 
\bigl\{ (u,\tilde{v}), (v,\tilde u) \mid (u,v) \in \maE \bigr\}  \:,
\end{equation}
\noindent where the set $\tilde \maV=\bigl\{\tilde{u} \mid u\in \maV \bigr\}$ is a disjoint copy of $\maV$. 
Also denote by $2\circ \I$, the set of independent sets of $2\circ G$, and for
all $U \subset 2\circ\maV$, let $2\circ \maE(U)$ be the set of neighbors of the elements of $U$ in $2\circ G$.  
The bipartite double cover of the graph of Example \ref{ex-theexample} 
is given in Figure \ref{fi-doubling-example}.

\begin{figure}[h!]
\vspace{-0.3cm}
\begin{center}
\begin{tikzpicture}
\draw[-] (0,3) -- (0,2);
\draw[-] (0,2) -- (-1,1);
\draw[-] (0,2) -- (1,1);
\draw[-] (-1,1) -- (1,1);
\fill (0,3) circle (2pt) node[right] {\small{1}} ;
\fill (0,2) circle (2pt) node[right] {\small{2}} ;
\fill (-1,1) circle (2pt) node[below] {\small{3}} ;
\fill (1,1) circle (2pt) node[below] {\small{4}} ;
\fill (2,2) node[] {$\longrightarrow$} ;
\fill (4,3) circle (2pt) node[above] {\small{1}} ;
\fill (5,3) circle (2pt) node[above] {\small{2}} ;
\fill (6,3) circle (2pt) node[above] {\small{3}} ;
\fill (7,3) circle (2pt) node[above] {\small{4}} ;
\fill (4,1) circle (2pt) node[below] {\small{$\tilde 1$}} ;
\fill (5,1) circle (2pt) node[below] {\small{$\tilde 2$}} ;
\fill (6,1) circle (2pt) node[below] {\small{$\tilde 3$}} ;
\fill (7,1) circle (2pt) node[below] {\small{$\tilde 4$}} ;
\draw[-] (4,3) -- (5,1);
\draw[-] (5,3) -- (4,1);
\draw[-] (5,3) -- (6,1);
\draw[-] (5,3) -- (7,1);
\draw[-] (6,3) -- (5,1);
\draw[-] (6,3) -- (7,1);
\draw[-] (7,3) -- (5,1);
\draw[-] (7,3) -- (6,1);
\end{tikzpicture}
\vspace{-0.3cm}
\caption[smallcaption]{The matching graph of Example \ref{ex-theexample} and its bipartite double cover.}
\label{fi-doubling-example}
\end{center}
\end{figure}
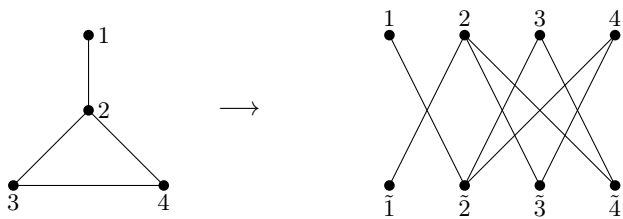
 
Consider a probability $\mu$ on $\maV$, and define the probability
$2\circ \mu \in \cM(2\circ \maV)$ by
\[
\forall u \in \maV, \qquad 2\circ \mu(u)=2\circ \mu(\tilde{u})=\mu(u)/2\:.
\]
Observe the following connection between the conditions {\sc Ncond}(.) and {\sc
  Ncond}$_{1/2}$(.),

\begin{lemma}
\label{lemma:g-2g}
For any graph $G$, we have   
\begin{equation}\label{eq-g-2g}
\bigl[ \mu \in \textsc{Ncond} (G) \bigr] \iff \bigl[ 2\circ \mu \in \textsc{Ncond}_{1/2}
(2\circ G) \bigr] \:.
\end{equation}
\end{lemma}

\medskip

\proof 
($\Longrightarrow$). Let $\maI \in 2\circ \I$. We can then write $\maI=A
\cup \tilde B$, $A\subset \maV, \ \tilde B \subset \tilde \maV$. Observe that the corresponding subset 
$A \cup B$ of $\maV$ is not an independent set of $G$ in general, because neither $A$ nor $B$ are so. 
But in view of Lemma \ref{lemma:equivNCOND}, we may write that 
\begin{align*}
2\circ\mu(\maI)=\mu(A)/2+\mu(B)/2 &< \mu\left(\maE(A)\right)/2+\mu\left(\maE(B)\right)/2\\
              &=2\circ \mu\left(2\circ \maE(A)\right)+2\circ \mu\left(2\circ \maE(\tilde B)\right)
              =2\circ \mu\left(2\circ \maE(\maI)\right),
\end{align*} 
where the last equality follows from the fact that $2\circ \maE(A)$ and $2\circ\maE (\tilde B)$ form a partition of $2\circ\maE(\maI).$\\ 
\medskip
\medskip
\noindent ($\Longleftarrow$). Let  $\maI \in \I$ and let $\tilde \maI$
be its copy in $\tilde \maV$. Clearly, 
$\maI\cup \tilde \maI \in 2\circ \I$, therefore 
\begin{align*}
\mu(\maI)=2\circ \mu\left(\maI \cup \tilde \maI\right) <2\circ\mu\left(2\circ\maE\left(\maI \cup \tilde \maI\right)\right)
 &=2\circ\mu\left(2\circ\maE\left(\maI\right)\right)+ 2\circ\mu\left(2\circ\maE\left(\tilde \maI\right)\right)\\
 &=\mu\left(\maE(\maI)\right)/2+\mu\left(\maE(\maI)\right)/2
  =  \mu\left(\maE(\maI)\right)\:. \,\,\Halmos  
\end{align*}
\endproof

\paragraph{{\bf Checking the conditions  \textsc{Ncond}.}}

Given $G$ and $\mu$, how to check efficiently
whether the conditions \textsc{Ncond}$(G)$ hold?  

The
cardinality of $\maI$ is exponential in $|\maV|$, so checking directly
all the inequalities yields an algorithm of
exponential time-complexity. But it is possible to do better. 

\begin{proposition}\label{pr-ncond}
Given a graph $G=(\maV,\maE)$ and a probability $\mu$ on $\maV$, there
exists an algorithm of time complexity $O(|\maV|^3)$ to decide if
$\mu$ satisfies \textsc{Ncond}$(G)$. 
\end{proposition}

\proof  
The result \cite[Prop. 3.5]{BGMa12} implies in particular that the
checking of $\text{\textsc{Ncond}}_{1/2}(2\circ G)$ can be done with an algorithm
of time complexity $O(|\maV|^3)$. 
Using Lemma \ref{lemma:g-2g}, we obtain the result for \textsc{Ncond}$(G)$ as a
direct corollary.\,\,\Halmos 
\endproof

\subsection{Main result}


\begin{theorem}\label{th-main}
Let $G$ be a connected graph. We have 
\begin{equation}\label{eq-sn1}
\bigl[
G \ \text{non-bipartite} \bigr] \iff \bigl[  \textsc{Ncond}(G) \neq \emptyset \bigr] \:.
\end{equation}
\end{theorem}


\proof 
Let $G$ be a connected graph. We first prove that 
\begin{equation}\label{eq-s1}
\bigl[
G \ \text{bipartite} \bigr] \implies \bigl[ \textsc{Ncond}(G) = \emptyset \bigr] \:.
\end{equation}
So suppose that $G$ is bipartite, and let $\maV=\maV_1\cup \maV_2$ be a
bi-partition of the vertices of $G$. Since $G$ is connected, we have
$\maE(\maV_1)=\maV_2$ and $\maE(\maV_2)=\maV_1$. The 
corresponding conditions in $\textsc{Ncond}(G)$ are
\begin{equation}\label{eq-contrad}
\mu(\maV_1) < \mu(\maV_2); \qquad \mu(\maV_2) < \mu(\maV_1) \:,
\end{equation}
hence (\ref{eq-s1}). 
The following implication is proved in 
\cite[Theorem 4.2]{BGMa12}: 
\begin{equation}\label{eq-BGMa12}
\bigl[
G \ \text{bipartite} \bigr] \implies \bigl[ \textsc{Ncond}_{1/2}(G) \neq \emptyset \bigr] \:.
\end{equation}
Consequently, comparing (\ref{eq-s1}) and \eref{eq-BGMa12} and using Lemma \ref{lemma:g-2g}, we see
that \eref{eq-contrad} is the only contradiction preventing
\textsc{Ncond}$(G)$ to hold whenever $G$ is connected and bipartite.

It remains to prove that 
\begin{equation}\label{eq-s2}
\bigl[G \ \text{non-bipartite} \bigr] \implies \bigl[\textsc{Ncond}(G) \ne \emptyset \bigr] .
\end{equation}
For this we first need to recall an auxiliary result. Consider a {\em directed}
bipartite graph $D= (\maV_1\cup \maV_2, \maE_1\cup \maE_2), \ \maE_1
\subset \maV_1\times \maV_2, \ \maE_2
\subset \maV_2\times \maV_1$. Given $\nu\in \cM(\maV_2)$,
define $\bar{\nu} \in \cM(\maV_1\cup \maV_2)$ by
\[
\forall u \in \maV_1, \ \bar{\nu}(u) = \nu(\maV_2\times \{u\})/2, \quad 
\forall u \in \maV_2, \ \bar{\nu}(u) = \nu(\{u\}\times \maV_1)/2 \:. 
\]
The next statement is a
direct consequence of \cite[Theorem 4.2]{BGMa12}: if $D$ is
strongly connected, then, since $G$ is connected and non-bipartite, the graph 
$$UD=(\maV_1\cup \maV_2, \maE_1\cup \{(v,u) \mid (u,v)\in
\maE_1\})$$
is itself connected. Thus, 
\begin{equation}\label{eq-BGMa12b}
\exists\,\nu \in \cM(\maV_2), \qquad \bar{\nu} \in
\textsc{Ncond}_{1/2}(UD)\:.
\end{equation}
Let us get back to the proof of (\ref{eq-s2}). 
The next result is standard and proved in \cite[Th. 3.4]{BHM80}:  if $G$ is connected, then
\[
[G \  \textrm{non-bipartite} ] \iff [ 2 \circ G \ \textrm{connected} ] \:.
\]
So assume that $G$ is connected and non-bipartite, then its bipartite double cover $2\circ G$ is connected. 
Consider the {\em directed} graph $D$ defined by
\[
\text{nodes:} \ 2\circ\maV= \maV \cup \tilde{\maV}, \qquad \text{arcs:} \
\{ u\rightarrow \tilde{v} \mid (u,v) \in \maE \} \cup
\{\tilde{u}\rightarrow u \mid u\in \maV\} \:.
\]
It is easy to prove that $D$ is strongly connected. 
Let us apply \eref{eq-BGMa12b} to $D$ with $\maE_2=
\{\tilde{u}\rightarrow u \mid u\in \maV\}$. We obtain the existence of 
$\bar{\nu}\in \textsc{Ncond}_{1/2} (2\circ G)$ and, by construction,
$\bar{\nu}(u) = \bar{\nu}(\tilde{u})$ for all $u\in \maV$. 
Therefore, according
to \eref{eq-g-2g}, the probability measure $\mu \in \cM(\maV)$ defined by
$$\mu(u) = \bar{\nu}(u) + \bar{\nu}(\tilde{u}),\,u \in \maV,$$  
belongs to $\textsc{Ncond}(G)$. This completes the proof. 
\Halmos
\endproof

\section{{\bf Stability of the matching model}}
\label{sec:iid}

To formalize the definition given in \S
\ref{sec:intro}, the matching model is specified by a triple
$(G,\Phi,\mu)$, where 
\begin{itemize}
\item $G=(\maV,\maE)$ is the matching
graph defined as in \S \ref{sec:intro}, and assumed to be connected. 
\item $\Phi$ is the matching policy defined as follows. 
We view the state of the buffer as a word over the alphabet $\maV$. 
More precisely, the state space is  
$$\maU=\Bigl\{u\in \maV^* \mid \forall (i,j)\in\maE, \ |u|_i\times
|u|_j=0 \Bigr\} \:$$
and we denote by $U_n \in \maU$, the state of the system just before the arrival 
of item $v_n$, for any $n\in\N$. 
The \emph{matching policy} is a mapping $\Phi : \maU \times  \maV \rightarrow
\maU$. In words, $\Phi(U,v)$ is the new buffer-content after the
arrival of an element $v$ in a buffer of content $U$. 
Observe that only the current state of the buffer is taken into account, which is a restriction, but a 
reasonable one.  
\item $\mu\in \cM(\maV)$ is the probability distribution of the
arrivals. Precisely, the sequence of
arriving items $(v_n)_{n\in \N}$ is i.i.d. of
common law $\mu$. 
\end{itemize}

Let $\mathbf 0$ be the empty word of $\maV^*$. Given a matching model $(G,\Phi,\mu)$ and a sequence of arrivals
$(v_n)_{n\in \N}$, 
the sequence of buffer-contents $(U_n)_{n\in \N}$ is a Markov
chain over the state space $\maU$ satisfying 
\[ 
U_0 = \mathbf 0, \qquad U_{n+1} = \Phi( U_n , v_n );\,n\in\N \:.
\]
This Markov chain is clearly irreducible and periodic of period 2.
 We say that the matching
model is {\em stable} if $(U_n)_{n\in \N}$ is positive recurrent. 

Consider the pair $(G,\Phi)$ formed by the
matching graph and the matching policy. 
The {\em stability region} of  $(G,\Phi)$ is the subset of $\cM(\maV)$
formed by the probability measures $\mu$ such that $(G,\Phi,\mu)$ is stable. 

\subsection{More on matching policies}

The matching policy may depend on the order of the items (i.e. on
their arrival dates). An
example is FCFS ("First Come, First Served"), where an arriving
item of class $j$ is matched with the oldest (if any) item of class $i$ in the buffer such that 
$j \v i$. 

Other matching policies are independent of the
arrival dates. In such cases, the matching decision at time $n$ depends
only on the commutative image $[U]$ of the state $U \in \maU$. In other words, the sequence
$([U_n])_{n\in \N}$ is a Markov chain on the state space  
\[
[\maU]=\Bigl\{u\in \N^{\maV} \mid \forall (i,j)\in\maE, \ u_i\times
u_j=0 \Bigr\} \:.
\]
Two such policies are considered below: 
``Match the longest'' and ``Priority''. 
For $i\in \maV$, let $e_i\in \N^{\maV}$ be defined by $(e_i)_i =1$
and $(e_i)_j=0, j\neq i$. 

\medskip


{\em Match the Longest} is the matching policy $\textsc{ML}:
[\maU]\times \maV \longrightarrow [\maU]$ defined by
\begin{equation}\label{eq-ml}
(U,i) \longmapsto \begin{cases} U + e_i & \ \mbox{if} \ \bigl[ j \in \maE(i)
  \implies U_j=0 \bigr]; \\
  U - e_j, \   j = \max \{ 
  \textsc{argmax} \ U_{|\maE(i)}
  \} & \ \mbox{otherwise},
\end{cases} 
\end{equation}
where $\textsc{argmax} \ U_{|\maE(i)}$ is the set of indices $k$ of $\maE(i)$ for which $U_k$ is positive and maximal. 
This set is non-empty and $j$ is the maximum with respect to some given total order on $\maV$. 
In words, ML gives priority to the more represented compatible class in the buffer. 

\medskip

Let us now define the priority policies. 
For each $i\in \maV$, define the {\em preferences} of $i$ as a total
order on the set $\maE(i)$. {\em Priority} is the matching policy $\Phi:
[\maU]\times \maV \longrightarrow [\maU]$ defined by
\begin{equation}\label{eq-prio}
(U,i) \longmapsto \begin{cases} U + e_i & \ \mbox{if} \ \bigl[ j \in \maE(i)
  \implies U_j=0 \bigr]; \\
  U - e_j, \   j = \max \{ \maE(i) \cap
  \mathrm{supp} \ U
  \} & \ \mbox{otherwise},
\end{cases} 
\end{equation}
where $\mathrm{supp} \ U = \{j \in \maV \mid U_j >0 \}$. In the above
second case, the set $\maE(i) \cap
  \mathrm{supp} \ U$ is non-empty and $j$ is the maximum with respect
  to the preferences of $i$ on $\maE(i)$.

\subsection{The results}
\label{subsec:posrec}

Our first result states that \textsc{Ncond} are
 necessary stability conditions. An analog
result holds for the bipartite matching model (see \cite[Lemma 3.2]{BGMa12}). 

\begin{proposition}
\label{thm:mainmono} 
Consider a connected matching graph $G$ and a matching policy
$\Phi$. We have for all $\mu \in \cM(\maV)$, 
$$(G,\Phi,\mu) \text{ stable }\Longrightarrow \mu \in \text{\textsc{Ncond}}(G).$$
\end{proposition}

\proof 
First assume that we have $\mu(\maI)>\mu\left(\maE(\maI)\right)$ for some independent set 
$\maI \subset \I$.  
For any $n\in\N$, let $\maI_n$ be the number of elements of $\maI$ in the system at 
time $n$, $F_n$ be the number of arrivals of type $\maI$ up to
time $n$, and $E_n$, the number of arrivals of type $\maE(\maI)$ up
to time $n$. Denote finally $H_n=F_n-E_n.$  
Observe that 
\begin{equation}
\label{eq:compareZ}
\left | U_n \right| \ge \maI_n\ge H_n \:. 
\end{equation}
By the strong law of large numbers,  we get 
\begin{equation}
\label{eq:LLN}
{\left | U_n \right| \over n} \ge {H_n \over n} \ \tend \
\mu\left(\maI\right)-\mu\left(\maE(\maI)\right) >0 \qquad \mbox{ a.s..}
\end{equation}
This implies that $(U_n)_{n}$ is transient. 

Suppose now that for some $\maI \in \I$, $\mu(\maI)=\mu\left(\maE(\maI)\right)$. 
In that case, the Markov chain $(H_n)_n$ is null recurrent. Again, in
view of (\ref{eq:compareZ}), the Markov chain
$(U_n)_n$ cannot be positive recurrent.\,\,\Halmos 
%
\endproof


\medskip

 The graph $G$ is said to be {\em separable of order $p$}, $p\ge 2$, if there exists a
 partition of $\maV$ into independent sets $\maI_1,\dots, \maI_p$, such that
 \[
 \forall i\neq j,\, \forall u \in \maI_i,\, \forall v \in \maI_j,\,\,u \v v \:. 
 \]

 \noindent In other words, $G$ is separable of order $p$ if its complementary
 graph can be partitioned into $p$ cliques. Notice that separable graphs of order 2 are bipartite 
(see Figure \ref{Fig:separable} below), whereas separable graphs of order 3 or more are non-bipartite. 

 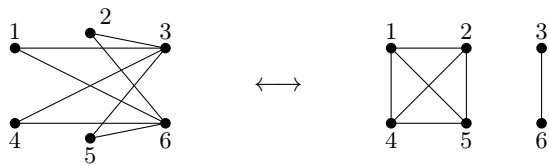
\begin{figure}[h!]
 \begin{center}
 \begin{tikzpicture}
 %
 \draw[-] (1,-1) -- (3,-2);
 \draw[-] (1,-1) -- (3,-1);
 \draw[-] (2,-0.8) -- (3,-2);
 \draw[-] (2,-0.8) -- (3,-1);
 \draw[-] (3,-1) -- (1,-2);
 \draw[-] (3,-1) -- (2,-2.2);
 \draw[-] (1,-2) -- (3,-2);
 \draw[-] (2,-2.2) -- (3,-2);
 \fill (1,-1) circle (2pt) node[above] {\small{1}} ;
 \fill (2,-0.8) circle (2pt) node[above right] {\small{2}} ;
 \fill (3,-1) circle (2pt) node[above] {\small{3}} ;
 \fill (1,-2) circle (2pt) node[below] {\small{4}} ;
 \fill (2,-2.2) circle (2pt) node[below] {\small{5}} ;
 \fill (3,-2) circle (2pt) node[below] {\small{6}} ;
 \fill (4.5,-1.5) node[] {$\longleftrightarrow$} ;
 \draw[-] (6,-1) -- (6,-2);
 \draw[-] (6,-1) -- (7,-2);
 \draw[-] (6,-1) -- (7,-1);
 \draw[-] (7,-1) -- (6,-2);
 \draw[-] (7,-1) -- (7,-2);
 \draw[-] (6,-2) -- (7,-2);
 \draw[-] (8,-1) -- (8,-2);
 \fill (6,-1) circle (2pt) node[above] {\small{1}} ;
 \fill (7,-1) circle (2pt) node[above] {\small{2}} ;
 \fill (8,-1) circle (2pt) node[above] {\small{3}} ;
 \fill (6,-2) circle (2pt) node[below] {\small{4}} ;
 \fill (7,-2) circle (2pt) node[below] {\small{5}} ;
 \fill (8,-2) circle (2pt) node[below] {\small{6}} ;
 \end{tikzpicture}
\vspace{-0.3cm}
 \caption{Separable graph of order $2$ (left) and the
   complementary graph (right).}
 \label{Fig:separable}
 \end{center}
 \end{figure}

\begin{theorem}\label{th-main-alea}
Consider a connected matching graph $G$. Let $\textsc{Pol}$ be the set
of matching policies. Let $\textsc{ML}$ be the
``Match the Longest'' policy, see \eref{eq-ml}. We have 
\begin{eqnarray}
G \text{ non-bipartite } & \iff & \exists \Phi \in \textsc{Pol}, \exists \mu \in \textsc{Ncond}(G), \
[(G,\Phi,\mu) \text{ stable }] \label{eq-non-bip}\\
G \text{ non-bipartite } & \implies & \forall \mu \in \textsc{Ncond}(G),
\ [(G,\textsc{ML},\mu) \text{ stable }] \label{eq-non-bip-ml}\\
\ \  G \text{ separable, }p\geq 3 & \implies & \forall \Phi \in \textsc{Pol}, \forall \mu \in \textsc{Ncond}(G), \
[(G,\Phi,\mu) \text{ stable }] \label{eq-sep}
\end{eqnarray}
\end{theorem}

\medskip

By merging Theorem \ref{th-main-alea} with the results from \cite{BGMa12}, we
get the following. A connected matching graph $G$ is either bipartite or
not. In the first case, we may construct a stable {\em bipartite}
matching model on $G$ (as in \cite{BGMa12}) but not a stable matching model. In
the second case, we may construct a stable matching model on $G$ (and the
bipartite matching model is not even defined). 

\proof[Proof of Theorem \ref{th-main-alea}] 
Fix the connected graph $G$. According to \eref{eq-sn1} in Theorem \ref{th-main}, the set $\textsc{Ncond}(G)$ is non-empty if and only if 
$G$ is non-bipartite. Therefore, we have 
\begin{align}
\exists \Phi \in \textsc{Pol}, \exists \mu \in \textsc{Ncond}(G), \
[(G,\Phi,\mu) \text{ stable }] &\Longrightarrow \textsc{Ncond}(G)\neq \emptyset\nonumber\\ 
                               &\Longrightarrow G \text{ non-bipartite}\:.\label{eq-interm}
\end{align}
Let us now prove that \eref{eq-non-bip-ml} holds. Together with
\eref{eq-interm}, it will also prove \eref{eq-non-bip}. 
 Let $\mu \in $ {\sc Ncond}$(G)$. Consider the bipartite double cover  
$2\circ G = (2 \circ \maV, 2 \circ \maE)$ of $G$. According
to Lemma \ref{lemma:g-2g}, we have $2\circ \mu \in \textsc{Ncond}_{1/2}
(2\circ G)$. 


Consider the bipartite matching model on the graph $2\circ G$ with
matching policy $ML$ and i.i.d. arriving sequence $(v_n, \widetilde v_n)_n$
of common law $2\circ \mu$. In the latter, let $\left([W_n], [\widetilde W_n]\right)_n$ be the corresponding buffer-content Markov Chain, as defined in (3) of \cite{BGMa12}, {\em i.e.} for any $i \in \maV$ and $\tilde i \in \widetilde{\maV}$, let $W_n(i)$ and 
$\widetilde W_n(\tilde i)$ count the number of buffered items of respective classes $i$ and $\tilde i$. 
Then, if $U_0=W_0$, one can easily check by induction that, for all $n$, we have $U_n=W_n$ almost surely. 
Since $2\circ \mu \in \textsc{Ncond}_{1/2}
(2\circ G)$, according to Theorem 7.1 in \cite{BGMa12}, the Markov
chain $\left([W_n], [\widetilde W_n]\right)_n$ is positive recurrent. We deduce that 
$([U_n])_n$ is also positive recurrent, which completes the proof of \eref{eq-non-bip-ml}. 

The only point that remains to be proved is \eref{eq-sep}.
 Assume that $G$ is separable of order $p \ge 3$ and let $\maI_1,\dots,
\maI_p$ be the independent sets partitioning $\maV$. For any $\maI \in \I$, there exists $i$ such that
$\maI\subset \maI_i$ and  $\maE(\maI)^c = \maI_i$. Therefore, 
\begin{equation}\label{eq-sn3}
\forall \maI \in \I, \ \maE(\maI)^c \in \I  \:.
\end{equation}
For $\maI\in\I$, define $\maI^\prime= \maE(\maI)^c$. 
Observe that $\maE(\maI) = \maE(\maI^\prime)$. In particular, we have 
\begin{equation}\label{eq-equiv}
\bigl[ \mu ( \maE(\maI)^c) <
\mu(\maE(\maI)) \bigr] 
\iff  \bigl[  \mu(\maI^\prime)   <
\mu(\maE(\maI^\prime)) \bigr] \:.
\end{equation}
Since $G$ is non-bipartite, $\textsc{Ncond}(G)$ is non-empty. 
Assume that $\mu \in \textsc{Ncond}(G)$ so that the right-hand side of \eref{eq-equiv} holds. Therefore, the left-hand side of
\eref{eq-equiv} holds as well for all $\maI\in
\I$. 
Consider the Lyapunov function $L$, defined for all $u \in \maU$ by
$L(u)=\left|u\right|.$ 
Fix $U_n=u\in \maU\setminus \{\mathbf 0\}$, and consider the
independent set $\maI^u=\bigl\{i\in\maV \ ; \  |u|_i >0\bigl\}.$
For any matching policy, the size of the buffer decreases (respectively, increases) at time $n+1$ if and only if 
$v_{n+1} \in \maE\left(\maI^u\right)$ (resp., $v_{n+1} \not\in \maE\left(\maI^u\right)$). Hence  
$$\esp{L\left(U_{n+1}\right)-L(u) \mid U_n=u}=\mu\left(\maE\left(\maI^u\right)^c\right)-\mu\left(\maE\left(\maI^u\right)\right)<0.$$
We conclude that the model is stable by applying the Lyapunov-Foster Theorem (see for instance
\cite[\S 5.1]{brem99}).\,\,\Halmos
\endproof

\section{{\bf Detailed study of the model of Example \ref{ex-theexample}}}

In this section, consider again the matching graph $G$ of Figure \ref{fi-example}. 
For simplicity, fix $\mu \in \mathcal M^+(\mathcal V)$ such that $\mu(3)=\mu(4)$. 
Let us fix a matching policy and denote by  $\textsc{Stab}$ the
stability region of the model. According to \eref{eq-ncond}, we have
$\textsc{Ncond}(G)=\{ \mu(1) < \mu(2) <
1/2 \}$. By Proposition \ref{thm:mainmono}, we have $\textsc{Stab} \subset \{ \mu(1) < \mu(2) <
1/2 \}$. Let us refine this statement with a non-trivial sufficient stability condition. 

\begin{lemma}\label{le-inclusions}
The stability 
region satisfies 
\begin{equation*}\label{eq-inclusions}
\textsc{Ncond}(G)\cap \bigl\{ \mu(1)(1-\mu(1)) <
\mu(2)^2 \bigr\}  \subset \textsc{Stab} \subset {\textsc{Ncond}(G)} \:.
\end{equation*}
\end{lemma}

\proof 
We only have to prove the left inclusion. 
Assume $\textsc{Ncond}(G)$ is satisfied. For $u$ in the state space
$\maU$, set $|u|_{34}=|u|_3+ |u|_4$. 
Fix $\eta$ such that
\begin{equation}
\label{eq:condeta}
\mu(1)\mu(2)^{-1}< 1 -\eta <1
\end{equation} and consider the Lyapunov function   
\[
L_\eta: \ \ \ 
\maU \longrightarrow \R_+, \qquad 
u \longmapsto (1-\eta) |u|_1 + |u|_2 + 
\mu(1)\mu(2)^{-1}\ |u|_{34} \:.
\]
Let us compute, for all $n\in\N$, 
$$\Delta_\eta = \esp{
  L_\eta\left(U_{n+1}\right)-L_\eta\left(U_n\right) \mid U_n = u } $$
in the different regions of the state space. If $|u|_2>0$, we have 
$$\Delta_\eta = \mu(2) - (1-\mu(2))=2\mu(2)-1,$$ 
so $\Delta_\eta < 0$ according to $\textsc{Ncond}(G)$. 
If $|u|_{34} >0$, we have 
\[
\Delta_\eta = (1-\eta)\mu(1) + \mu(3) \mu(1)\mu(2)^{-1} - \mu(3)\mu(1)\mu(2)^{-1}
- \mu(2) \alpha = (1-\eta) \mu(1)  - \mu(2) \alpha \:,
\]
where $\alpha=1-\eta$ if the arriving item of type 2 is matched with a buffered
item of type 1, and $\alpha=\mu(1)\mu(2)^{-1}$ otherwise. From (\ref{eq:condeta}), 
we get
\[
\Delta_\eta  \leq (1-\eta) \mu(1) - \mu(2)
\mu(1)\mu(2)^{-1} =-\eta \mu(1) <0 \:. 
\]
If $|u|_1>0$ and $|u|_{34} =0$, we have 
\[
\Delta_\eta = (1-\eta) \mu(1) +
2\mu(3)\mu(1)\mu(2)^{-1} - (1-\eta) \mu(2)\:.
\]
Replacing $2\mu(3)$ by $[1-\mu(1)-\mu(2)]$ and symplifying, we get
\[
\bigl[ \Delta_\eta <0 \bigr] \iff \bigl[ \eta\mu(2)(\mu(2)-\mu(1)) + \mu(1)(1-\mu(1)) < \mu(2)^2
\bigr]\:.
\]
Applying again the Lyapunov-Foster Theorem to the subset $A=\left\{\mathbf 0\right\}$, the model is stable on any 
region $\textsc{Ncond}(G) \cap \{ \eta\mu(2)(\mu(2)-\mu(1)) +
\mu(1)(1-\mu(1)) < \mu(2)^2 \}$, for $\eta$ satisfying (\ref{eq:condeta}). 
By letting $\eta$ go to 0, we obtain the left inclusion of Lemma \ref{le-inclusions}.\,\,\Halmos
\endproof

\medskip

According to Theorem \ref{th-main-alea}, the ML policy has a maximal stability region
and reaches the right bound in Lemma
\ref{le-inclusions}. It is then natural to wonder, whether there
exists a matching policy with the smallest possible stability region,
that is, reaching the left bound in Lemma \ref{le-inclusions}. 
To investigate this question, let us introduce two matching
policies of the {\em priority} type, see \eref{eq-prio}: 
\begin{itemize}
\item $A$: 2  gives priority to ``3 or 4'' over 1. $B$: 2 gives priority to 1 over ``3 or 4''. 
\end{itemize}
\noindent Denote by $\textsc{Stab}(A)$ and $\textsc{Stab}(B)$, the stability
regions of policy $A$ and $B$, respectively. 

\medskip

We use a simplified state space description $\breve{\maU}$, by considering the commutative
image of the states and by 
merging items 3 and 4~: 
\begin{equation*}
\breve{\maU} = \bigl\{ (0,\ell,0), \ell\in \N \bigr\} \cup \bigl\{ (k,0,m),\, k,m \in \N
\bigr\}=\breve{\maU}_2 \cup \breve{\maU}_{134}.
\end{equation*}
The buffer-content is described by the $\breve{\maU}$- valued Markov chain $(\breve U_n)_n$, 
where, 
$$\breve U_n(1)= |U_n|_1,\,\breve U_n(2)=|U_n|_2,\,\breve
U_n(3)=|U_n|_{34}.$$

Observe that $(\breve U_n)_n$ has to go through state $(0,0,0)$ to go from $\breve{\maU}_2$
to $\breve{\maU}_{134}$ (or the other way around). Due to this
property, $(\breve U_n)_n$ is positive recurrent iff the induced Markov chains on  $\breve{\maU}_2$
and $\breve{\maU}_{134}$ are both positive recurrent. 

\medskip

Let us consider first the induced Markov chain on  $\breve{\maU}_2$. It is the
same for the two priority policies and its transition matrix $P$
satisfies
\[
\forall i \in \N\setminus \{0\}, \qquad P_{i, i-1} = 1 -\mu(2), \quad
P_{i,i+1} = \mu(2) \:.
\]
So the stability condition of the induced chain is: $\bigl(\mu(2)< 1- \mu(2)\bigr) \iff \bigl( 
\mu(2)<1/2 \bigr).$

Now consider the induced Markov chains on $\breve{\maU}_{134}$, which depend on the
priority policy. The two induced chains are random walks on
$\Z^2_+$, meaning that the transition probabilities are homogeneous in
the interior of the state space, and along each of the axis. 
Denote by $Q_A$ and $Q_B$ the transition matrices of
the induced chains under the
policies $A$ and $B$ respectively. The graphs of $Q_A$ and $Q_B$ are
represented in Figure \ref{fi-trgraph}, where $(i,j)$ corresponds to
the state $(i,0,j)$. 

\begin{figure}[htb]
 \begin{center}
 \begin{tikzpicture}
\draw[->] (0,0) -- (4,0) node[right]{\scriptsize{items 1}};
\draw[->] (0,0) -- (0,4) node[above]{\scriptsize{items 3 or 4}};
\fill (2,0) circle (2pt);
\draw[->, thick] (2,0) -- (2.5,0) node [above right]{\scriptsize{$\mu(1)$}};
\draw[->, thick] (2,0) -- (1.5,0) node [above left]{\scriptsize{$\mu(2)$}};
\draw[->, thick] (2,0) -- (2,0.5) node [above]{\scriptsize{$2\mu(3)$}};
%
\fill (3,3) circle (2pt);
\draw[->, thick] (3,3) -- (3.5,3) node [right]{\scriptsize{$\mu(1)$}};
\draw[->, thick] (3,3) -- (3,3.5) node [above]{\scriptsize{$\mu(3)$}};
\draw[->, thick] (3,3) -- (3,2.5) node [below]{\scriptsize{$\mu(2)+\mu(3)$}};
%
\fill (0,2) circle (2pt);
\draw[->, thick] (0,2) -- (0.5,2) node [right]{\scriptsize{$\mu(1)$}};
\draw[->, thick] (0,2) -- (0,2.5) node [right]{\scriptsize{$\mu(3)$}};
\draw[->, thick] (0,2) -- (0,1.5) node [right]{\scriptsize{$\mu(2)+\mu(3)$}};
%
%
%
\draw[->] (7,0) -- (11,0) node[right]{\scriptsize{items 1}};
\draw[->] (7,0) -- (7,4) node[above]{\scriptsize{items 3 or 4}};
\fill (9,0) circle (2pt);
\draw[->, thick] (9,0) -- (9.5,0) node [above right]{\scriptsize{$\mu(1)$}};
\draw[->, thick] (9,0) -- (8.5,0) node [above left]{\scriptsize{$\mu(2)$}};
\draw[->, thick] (9,0) -- (9,0.5) node [above]{\scriptsize{$2\mu(3)$}};
%
\fill (10,3) circle (2pt);
\draw[->, thick] (10,3) -- (10.5,3) node [right]{\scriptsize{$\mu(1)$}};
\draw[->, thick] (10,3) -- (10,3.5) node [above]{\scriptsize{$\mu(3)$}};
\draw[->, thick] (10,3) -- (9.5,3) node [left]{\scriptsize{$\mu(2)$}};
\draw[->, thick] (10,3) -- (10,2.5) node [below]{\scriptsize{$\mu(3)$}};
%
\fill (7,2) circle (2pt);
\draw[->, thick] (7,2) -- (7.5,2) node [right]{\scriptsize{$\mu(1)$}};
\draw[->, thick] (7,2) -- (7,2.5) node [right]{\scriptsize{$\mu(3)$}};
\draw[->, thick] (7,2) -- (7,1.5) node [right]{\scriptsize{$\mu(2)+\mu(3)$}};
\end{tikzpicture}
 \caption{The graph of $Q_A$ (left), and that of $Q_B$ (right).}\label{fi-trgraph}
 \end{center}
 \end{figure}
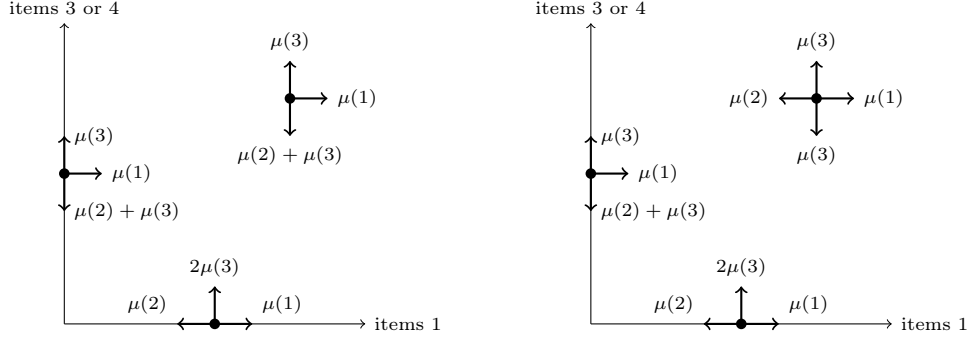
Let us justify, for instance, the coefficients 
$(Q_A)_{ij,i(j-1)}= \mu(2)+\mu(3),\,i\ge 0,\,j>0.$
In state $(i,j)$, there are either
$j$ items of type 3 or $j$ items of type 4. In the first case 
(resp. second case), one of the $j$ items is removed if an item of type 4 (resp., of type 3) arrives. 
In both cases, such an event occurs with the same probability $\mu(3)=\mu(4)$. Further, due to the
priority policy, one of the $j$ items is also 
removed whenever an item of type 2 arrives (probability $\mu(2)$).

The detailed study of random walks in $\Z^2_+$ is carried out in the
monograph \cite{FMM}. The salient result \cite[Theorem 3.3.1]{FMM},
is the necessary and sufficient condition for positive recurrence in terms of the one-step drifts of the random walk on the interior
of the quadrant, and on each of the axes. It applies directly to our context.

Let us first consider policy $A$. 
The drifts of the Markov chain are
\begin{eqnarray*}
\mathrm{Interior:} & D_x = \mu(1), & D_y = -\mu(2) \\
\mathrm{First \ axis:} & D_x' = \mu(1)-\mu(2), & D_y'= 2\mu(3) \\
\mathrm{Second \ axis:} & D_x'' = \mu(1), &  D_y'' = -\mu(2)\:.
\end{eqnarray*}

Since $D_x>0$ and $D_y<0$, the
Markov chain is stable iff $[D_xD_y' - D_yD_x' <0]$, see
\cite[Theorem 3.3.1]{FMM}. We have
\begin{equation*}
\bigl[ D_xD_y' - D_yD_x' <0 \bigr]\!\!\iff \!\!\bigl[ 2\mu(1)\mu(3) +
\mu(2)(\mu(1)-\mu(2)) <0 \bigr]\!\!\iff  \!\!\bigl[\mu(1)(1-\mu(1)) <
\mu(2)^2 \bigr] \:.
\end{equation*}



We now turn to the priority policy $B$. The drifts of the Markov chain read 
\begin{eqnarray*}
\mathrm{Interior:} & D_x = \mu(1)-\mu(2), & D_y = 0 \\
\mathrm{First \ axis:} & D_x' = \mu(1)-\mu(2), & D_y'= 2\mu(3) \\
\mathrm{Second \ axis:} & D_x'' = \mu(1), &  D_y'' = -\mu(2)\:.
\end{eqnarray*}
\noindent Since $D_x<0$ and $D_y=0$, the
Markov chain is stable iff $[D_yD_x'' - D_xD_y'' <0]$, see
\cite[Theorem 3.3.1]{FMM}. We have 
\begin{equation*}
\bigl[ D_yD_x'' - D_xD_y'' <0 \bigr] \iff  \bigl[
\mu(2)(\mu(1)-\mu(2)) <0 \bigr] \iff \bigl[
\mu(1) < \mu(2) \bigr] \:.
\end{equation*}
\noindent Summarizing all of the above, we get the next proposition.

\begin{proposition}\label{pr-stabAB}
The stability regions under policies $A$ and $B$ are respectively: 
\begin{equation*}
 \textsc{Stab}(A) = {\textsc{Ncond}(G)}\cap \bigl\{ \mu(1)(1-\mu(1)) <\mu(2)^2 \bigr\};\,\,\,\quad  
\textsc{Stab}(B) = {\textsc{Ncond}(G)}\:.
\end{equation*}
\end{proposition} 

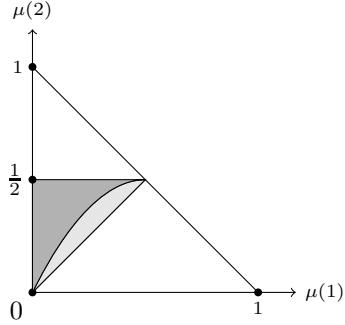
\begin{figure}[h!]
\begin{center}
\begin{tikzpicture}
\fill (0,0) circle (1.5pt) node[below left]{0} ;
\draw[->] (0,0) -- (0,3.5);
\fill (0,3.5) node[above]{\scriptsize{$\mu(2)$}} ;
\fill (0,3) circle (1.5pt) node[left]{\scriptsize{$1$}};
\draw[->] (0,0) -- (3.5,0);
\fill (3.5,0) node[right]{\scriptsize{$\mu(1)$}} ;
\fill (3,0) circle (1.5pt) node[below]{\scriptsize{$1$}};
\draw[-] (3,0) -- (0,3);
\draw[-] (0,0) -- (1.5,1.5);
\filldraw [fill=gray!20,draw=black]
(0,0) parabola[bend at end] (1.5,1.5) -- cycle;
\filldraw [fill=gray!60,draw=black]
(0,0) parabola[bend at end] (1.5,1.5) -- (0,1.5) -- cycle;
\fill (0,1.5) circle (1.5pt) node[left]{${1\over 2}$};
\end{tikzpicture}
\caption[smallcaption]{$\textsc{Stab}(A)$ is  the dark zone; $\textsc{Stab}(B)$ is the union of 
the dark and light zones.}\label{fi-zone}
\end{center}
\end{figure}

Using Lemma \ref{le-inclusions}, we can rephrase Proposition \ref{pr-stabAB} by saying that
policy $A$ has the smallest possible stability region, while policy
$B$ has the largest possible stability region. The two stability
regions are represented in Figure \ref{fi-zone}.

\end{document}